\documentclass[reqno, 8pt]{amsart}

\usepackage{amscd}
\input{xypic}

\theoremstyle{plain}
\newtheorem{thm}{Theorem}[section]
\newtheorem{prop}[thm]{Proposition}

\newtheorem{lemma}[thm]{Lemma}
\newtheorem{cor}[thm]{Corollary}
\newtheorem*{thm*}{Theorem}
\newtheorem*{prop*}{Proposition}
\theoremstyle{definition}
\newtheorem{defn}[thm]{Definition}

\theoremstyle{remark}

\newtheorem{remark}{Remark}

\newcommand{\Z}{\mathbb Z}
\newcommand{\Q}{\mathbb Q}

\newcommand\rar[1]{\stackrel{#1}{\longrightarrow}}

\newcommand{\op}{\operatorname}

\newcommand\F{\mathbb F}

\newcommand\sub{\subseteq}

\newcommand\EXT{\op{\mathcal{X}}}
\newcommand{\vv}{\,\vert\,}
\newcommand{\be}{\bold E}
\newcommand{\id}{\text{id}}

\title{Automorphisms and Cohomology}
\author{James A. Schafer}

\date{\today}

\address{\begin{flushleft}\quad Department of Mathematics \\
\quad University of Maryland \\
\quad College Park, Maryland 20742
\end{flushleft}}

\email{jas@math.umd.edu}

\begin{document}
\maketitle
\section{Introduction}

Let $1\to H\to G\to Q\to 1$ be an exact sequence of groups.  In \cite{O} the following exact sequence was developed for centric extensions, i.e the centralizer of $H$ in $G$ is contained in $H$,

$$ 0\to H^1(Q,zH)\rar{\mu}\op{Aut}(G,H)\rar{res}N_{\op{Out}H}(\Phi Q)/\Phi Q\rar{\lambda_\be}H^2(Q,zH)$$
where $\op{Aut}(G,H)$ are the automorphisms of $G$ which restrict to an automorphism of $H$, $\Phi:Q\to \op{Out} H$ is the outer action determined by the extension, $zH$ is the center of $H$ with $Q$-action coming form $\Phi$ and $N_{\op{Out} H}$ the normalizer.  

It is the aim of this paper to generalize the above sequence to arbitrary extensions, show how the above result is derived from the general exact sequence and derive other consequences of the general result including determing solvability of $\op{Aut}(G,H)$.

\section{Extending automorphisms}

Let $\bold E: 1\to H\rar{i} G\rar{\pi} Q\to 1$ be an exact sequence of groups.  Let $\alpha:H\to H$ and $\beta:Q\to Q$ be homomorphisms.  We are interested in the problem of deciding if there exists a homomorphism $\gamma:G\to G$ such that the diagram
$$\begin{CD}\bold E:\quad 1 @>>> H @>i>> G @>\pi>> Q @>>> 1 \\
@.      @V\alpha VV	@V\gamma VV  @V\beta VV   @.\\
\bold E:\quad 1 @>>> H @>i>> G @>\pi>> Q @>>> 1  \end{CD}$$ commutes.

The extension $\bold E$ gives rise to a homomorphism $\Phi:Q\to \op{Out}(H)$ where $\Phi(q)$ is the class of conjugation in $G$ of any preimage of $q$ and this determines a well defined action of $Q$ on $zH$ the center of $H$.  We will denote the image of $\varphi\in\op{Aut}(H)$ in $\op{Out}(H)$ by $[\varphi]$.  In any calculations we will write the group structure aditively in $H$ and $G$ and multiplicatively in $Q$.

\begin{prop} Suppose $\alpha\in\op{Aut}(H)$ then if $\gamma$ exists the following commutes.
$$\begin{CD} Q @>\Phi>> \op{Out}(H)  \\
	@V\beta VV		@VV c_{[\alpha]} V\\
	Q @>\Phi >> \op{Out}(H)  \end{CD}$$
where $c_{[\alpha]}$ denotes conjugation by $[\alpha]$ mapping $\op{Out}(H)\to \op{Out}(H')$.\end{prop}
\begin{proof} If $u:Q\to G$ is a section then $c_{[\alpha]}\Phi(q)=[\alpha\circ c_{uq}\circ\alpha^{-1}]$.  If $n\in H$ then $\alpha\circ c_{uq}\circ\alpha^{-1}(n)=\alpha[uq+\alpha^{-1}-uq]$ but $\alpha$ is the restriction of $\gamma$ and so this is $\gamma[uq+\alpha^{-1}-uq]=\gamma(uq)+n-\gamma(uq)=c_{\gamma(uq)}(n)$.  Therefore $c_{[\alpha]}\Phi(q)=[c_{\gamma(uq)}]$.  But $\pi(\gamma(uq))=\beta(\pi(uq)=\beta\,q$ and so $[c_{\gamma(uq)}]=\Phi\circ\beta(q)$.\end{proof}

We shall now assume $\alpha$ as an automorphism and that the diagram
$$\begin{CD} Q @>\Phi>> \op{Out}(H)  \\
	@V\beta VV		@VV c_{[\alpha]} V\\
	Q @>\Phi >> \op{Out}(H)  \end{CD}$$
commutes.  




\section{Pullbacks and pushouts of extensions}

Let $\op{\mathcal{X}}_\Phi(Q,H)$ equivalence classes of extensions $\bold E: 1\to H\rar{i} G\rar{\pi} Q\to 1$ with associated homomorphism $Q\to\op{Out}(H)$ equal to $\Phi$.

If $\beta:Q'\to Q$ then we have a pullback extension and maps
$$\begin{CD}\beta^*\bold E:\qquad  1 @>>> H @>i'>> X @>\pi'>> Q' @>>> 1 \\
@.      @\vert	@V\gamma VV  @V\beta VV   @.\\
\,\bold E:\qquad 1 @>>> H @>i>> G @>\pi>> Q @>>> 1  \end{CD}$$ where $X$ is the pulback of $\pi$ and $\beta$.

It is quite easy to see that if $\bold E\equiv \bold E'$ then $\beta^*\bold E\equiv \beta^*\bold E'$.

\begin{prop} $\Phi'=\Phi\beta:Q'\to\op{Out}(H)$ and hence $\beta^*:\op{\mathcal{X}}_\Phi(Q,H)\to \op{\mathcal{X}}_{\Phi\beta}(Q,H)$.\end{prop}

\begin{proof} First assume $\beta$ is a monomorphism. If $u:Q\to G$ is a section, then $u'=u|Q':Q'\to G'\sub G$ is a section and $\Phi'(q')=[c_{u'(q')}]=[c_{u(q')}]=\Phi(q')$.  If $\beta$ is an epimorphism and $u':Q'\to X$ is a section (all sections have $u'(1)=0$) then define $u:Q\to G$ as follows.  Choose a normalized section $s:Q\to Q'$ and define $u(q)=\gamma\circ u'(sq)$.  Note $\pi\,u(q)=\beta\,\pi'\,u'(s\,q)=q$.   
For $q'\in Q'$, $\Phi'(q')=[c_{u'q'}]$ while $\Phi\beta(q')=[c_{u(\beta q')}]=[c_{\gamma(u's\beta q)}]$.  If $\gamma(u'q')$ and $\gamma(u's\beta q )$ are both preimages of $\beta q'$ and so $[c_{\gamma(u'q')}]=[c_{\gamma(u'\bar q)}]$.  But since $\gamma\vert H=\text{Id}$, $c_{\gamma(u'q')}=c_{u'q'}:H\to H$ for all $q'\in Q'$.\end{proof}

If $\bold E: 1\to H\rar{i} G\rar{\pi} Q\to 1$ with homomorphism $\Phi:Q\to\op{Out}(H)$ then by choosing a (normalized) section $u:Q\to H$, define $\varphi(q)=c_{uq}$ and $f:Q\times Q\to H$ by
$$  u(q)+u(q')=f(q,q')+u(qq')  $$ one sees
\begin{enumerate}\item $[\varphi(q)]=\Phi(q)$
\item $\varphi(q)\varphi(q')=c_{f(q,q')}\varphi(qq')$,
\item  $\varphi(q)f(q',q'')+f(q,q'q'')= f(q,q')+ f(qq',q'')$,\end{enumerate}
and that the extension $\bold E$ is equivalent to $\bold E_{f,\varphi}:\,1\to H\rar{i} E_{f,\varphi}\rar{\pi} Q\to 1$ where $E_{f,\varphi}=H\times Q$ with multiplication $(n,q)+(n',q')=(n+\varphi(q)n'+f(q,q'),qq')$ and obvious maps $i$ and $\pi$ (and section).

\begin{prop} If $\bold E\equiv\bold E_{f,\varphi}$ then $\beta^*\bold E\equiv\bold E_{f^\beta,\varphi^\beta}=\bold E_{f(\beta\times\beta),\varphi\beta}$.\end{prop}
\begin{proof} If $u:Q\to G$ is a normalized section for $\bold E$ such that $\varphi(q)=c_{u(q)}$, define $u':Q\to X$ by $u'(q)=(u(\beta q),q)\in X\sub G\times Q$. $u'$ is a normalized section and $c_{u'q}(n,1)=(\varphi(\beta q)n,1)$ and $u'q+u'q'-u'(qq')= (f(\beta q,\beta q'),1)$.

\end{proof} 


We would like the notion of pushout corresponding to pullback but these do not exist in general in the category of groups.

\medskip Now suppose $\alpha:H\to H$ is an automorphism and $\bold E\equiv\bold E_{f,\varphi}$. Define $\alpha_*\bold E_{f,\varphi}=\bold E_{f_{\alpha},\varphi_{\alpha}}$ where $f_\alpha=\alpha f$ and $\varphi_\alpha=c_{\alpha}\varphi$, that is $\varphi_\alpha(q)=\alpha\varphi(q)\alpha^{-1}$.  It is easily seen that $f_\alpha$ and $\varphi_\alpha$ satisfy (2) and (3) and that $\Phi_\alpha=c_{[\alpha]}\Phi:Q\to\op{Out}(H)$.  In order to show we may define $\alpha_*\bold E$ as $\alpha_*\bold E_{f,\varphi}$ we need to show that if $\bold E_{f,\varphi}\equiv\bold E_{f',\varphi'}$ then $\bold E_{f_\alpha,\varphi_\alpha}\equiv\bold E_{f'_\alpha,\varphi'_\alpha}$.  
 
We will use the following convention.  If $\bold E: 0\to H\rar i G\rar \pi Q\to 1$ and $\bold E':0\to H'\rar{i'} G'\rar{\pi'} Q'\to 1$ are extensions and $\alpha:H\to H'$, $\beta:Q\to Q'$, we will write $(\alpha,\beta):\bold E\to\bold E'$ if there exists $\gamma:G\to G'$ such that $\gamma i=i'\alpha$ and $\pi\beta=\pi'\gamma$.  By abuse of notation we will denote the set of such $\gamma$ by $(\alpha,\beta)$.

The following lemma will prove very useful.

\begin{prop}\label{map} Suppose $$\begin{CD} 1 @>>> H @>i>> E_{f,\varphi} @>\pi>> Q @>>> 1 \\
@.      @V\alpha VV	@.  @V\beta VV   @.\\
 1 @>>> H' @>i>> E_{f',\varphi'} @>\pi>> Q' @>>> 1.\end{CD}$$
Then there exists $\gamma:E_{f,\varphi}\to E_{f',\varphi'}$ with $\gamma i=i\alpha$ and $\pi\gamma=\beta\pi$ if and only if there exists $\sigma:Q\to H'$ such that for all $h\in H\, q,q'\in Q$
\begin{equation} \sigma q+\varphi'(\beta q)[\alpha h+\sigma q']+f'(\beta q,\beta q')=\alpha[\varphi(q)h]+\alpha f(q,q')+\sigma(qq').\end{equation} 

(a) The correspondence between $\{\gamma\vv \gamma\,i=i'\alpha, \gamma\pi=\pi'\beta\}$ and $\{\sigma\vv \sigma satisfies (1)\}$ is a bijection.
\par (b) If $\gamma$ corresponds to $\sigma$ and $\gamma'$ corresponds to $\sigma'$, then $\gamma'\gamma:E_{f,\varphi}\to E_{f'',\varphi''}$ corresponds to $\alpha'\sigma+\sigma'\beta:Q\to H''$.
\end{prop} 
\begin{proof} That $\gamma(h,q)=(\alpha h+\sigma q,\beta q)$ is forced by $\gamma i=i\alpha$ and $\pi\gamma=\beta\pi$.
Condition (1) is just the fact that $\gamma$ is to be a homomorphism.  (a) and (b) are immediate from the relationship between $\gamma$ and $\sigma$.\end{proof}

\begin{prop} (a) There exists $(\alpha,\id):\be_{f,\varphi}\to\alpha_*\be_{f,\varphi}$.
\par (b) If $(\alpha,\beta):\be_{f,\varphi}\to\be_{f',\varphi'}$ then there exists a unique $(\id,\beta):\alpha_*\be_{f,\varphi}\to\be_{f',\varphi'}$ such that 
$$(\alpha,\beta)=(\id,\beta)(\alpha,\id):\be_{f,\varphi}\to\alpha_*\be_{f,\varphi}\to\be_{f'\varphi'}.$$\end{prop}
\begin{proof} (a) $\sigma\equiv 0$ defines a $\be_{f,\varphi}\to\alpha_*\be_{f,\varphi}$ since $\varphi_\alpha(\alpha h)=\alpha\varphi(h)$.
\par (b). Suppose $\gamma\in(\alpha,\beta)$ corresponds to the map $\sigma:Q\to H'$, that is
$$ \sigma q+\varphi'(\beta q)[\alpha h+\sigma q']+f'(\beta q,\beta q')=\alpha[\varphi(q)h]+\alpha f(q,q')+\sigma(qq').$$
But the right hand side of the equation is $\varphi_\alpha(\alpha h)+f_\alpha(q,q')+\sigma(qq')$ and hence letting $h'=\alpha h$, $\sigma$ defines a map $(\id,\beta):\alpha_*\be_{f,\varphi}\to\be_{f',\varphi'}$.  By (b) of the previous proposition $(\id,\beta)(\alpha,\id)$ corresponds to $\alpha\,0+\sigma\,\id=\sigma$ and so $(\alpha,\beta)=(\id,\beta)(\alpha,\id)$.\end{proof}

\begin{cor} If $\bold E_{f,\varphi}\equiv\bold E_{f',\varphi'}$ then $\alpha_*\bold E_{f,\varphi}\equiv\alpha_*\bold E_{f',\varphi'}$.  It follows that $\alpha_*\bold E$ is well defined.\end{cor}
\begin{proof} If $(\id,\id):\be_{f,\varphi}\to\be{f',\varphi'}$ is an equivalence then the map $(\alpha,\id)(\id,\id):\be_{f,\varphi}\to \alpha_*\be_{f'\varphi'}$ factors through $\alpha_*\be_{f,\varphi}$.  That is there exists $(\rho,\tau):\alpha_*\be_{f,\varphi}\to\alpha_*\be_{f'\varphi'}$ with
$(\rho,\tau)(\alpha,\id)=(\alpha,\id)(\id,\id):\be_{f,\varphi}\to\alpha_*\be_{f'\varphi'}$.  It follows $\rho$ and $\tau$ are the identities and define an equivalence of $\alpha_*\be_{f,\varphi}$ and $\alpha_*\be_{f'\varphi'}$.
\end{proof}

\begin{cor} Let $\alpha:H\to H$ be an isomorphism.
\begin{enumerate}\item There exists $(\alpha,1):\bold E\to \alpha_*\bold E$.
\item If $(\alpha,\beta):\be\to\be'$ then there exists $(\text{id},\beta):\alpha_*\be\to\be'$ such that $$(\alpha,\beta)=(\text{id},\beta)(\alpha,\text{id}):\be\to\alpha_*E\to\be'.$$
That is, $\alpha_*\be$ is a pushout in the category of extensions.
\item If there exists $(\alpha,1):\bold E\to \bold E'$, then $\bold E'=\alpha_*\bold E$.\end{enumerate}\end{cor}



\begin{prop}\begin{enumerate}\item If $\alpha:H\to H'$ and $\alpha':H'\to H''$ are isomorphisms then $(\alpha'\alpha)_*\bold E\equiv \alpha'_*(\alpha_*\bold E).$
\item If $\beta:Q\to Q'$ and $\beta':Q'\to Q''$ then $(\beta'\beta)^*\bold E\equiv \beta^*(\beta'^*\bold E).$
\item If $\alpha:H\to H'$ is an isomorphism and $\beta:Q\to Q'$ then $\alpha_*\beta^*\bold E\equiv \beta^*\alpha_*\bold E$. \end{enumerate} \end{prop}
\begin{proof} These are immediate from the definition of $\alpha_*\bold E\equiv\alpha_*\bold E_{f,\varphi}$ and similarly for $\beta^*$.\end{proof}


\begin{thm}\label{conj} Suppose $\alpha=c_{\bar g}\vert H:H\to H$ with $\pi(g)\in zQ$, in particular if $g\in H$. Then $\bold E\equiv\alpha_*\bold E$. \end{thm}
\begin{proof} 
$c_g:G\to G$ defines a map $(\alpha,\id):\be\to \be$ and hence $\alpha_*\be\equiv\be$.
\end{proof}

\begin{thm}\label{MT} Suppose $\alpha$ and $\beta$ are automorphisms. Consider the following diagram

$$\begin{CD}\bold E:\quad 1 @>>> H @>i>> G @>\pi>> Q @>>> 1 \\
@.      @V\alpha VV	@.  @V\beta VV   @.\\
\bold E:\quad 1 @>>> H @>i'>> G @>\pi'>> Q @>>> 1  \end{CD}$$
with $\alpha$ an automorphism and suppose 
$$\begin{CD} Q @>\Phi>> \op{Out}(H)  \\
	@V\beta VV		@VV c_{[\alpha]} V\\
	Q @>\Phi >> \op{Out}(H)  \end{CD}$$ commutes. 
Then there exists $(\alpha,\beta):\be\to \be$ if and only if
$\alpha_*\bold E\equiv\beta^*\bold E\in\op{\mathcal{X}}_{\Phi\beta}(Q,H)$\end{thm}

\begin{proof} If $(\id,\id):\alpha_*\be\to\beta^*\be$ then we have $(\alpha,\beta):\be\to\be$ given by $$(\id,\beta)(\id,\id)(\alpha,\id):\be\to\alpha_*\be\to\beta^*\be\to\be.$$
Conversely since there exists $(\alpha,\id):\be\to\alpha_*\be$ there exists $(\alpha^{-1},\id):\alpha_*\be\to\be$ and similarly for $\beta$.  Hence if $(\alpha,\beta):\alpha_*\be\to\beta^*\be$ exists,
$$(\id,\beta^{-1})(\alpha,\beta)(\alpha^{-1},\id):\alpha_*\be\to\be\to\be\to\beta^{-1}\be$$ is an equivalence $(\id,\id):\alpha_*\be\to\beta^*\be$.\end{proof}

\section{The action of $\mathcal S$ on $\mathcal X_\Phi(Q,H)$ and $H^2(Q,zH)$}



\begin{defn}$\mathcal S=\{(\alpha,\beta)\in\op{Aut}(H)\times\op{Aut}(Q)\,\vert\, \Phi\,\beta=c_{[\alpha]}\Phi]\}$. \end{defn}
 This is clearly a subgroup.  Since $\alpha_*:\EXT_\Phi(Q,H)\to\EXT_{c_{[\alpha]}\Phi}(Q,H)$ and $\beta^*:\EXT_\Phi(Q,H)\to\EXT_{\Phi\beta}(Q,H)$, if $\theta=(\alpha,\beta)\in\mathcal S$ then $\alpha^{-1}_*\beta^*:\EXT_\Phi(Q,H)\to\EXT_\Phi(Q,H)$.

\begin{prop} If $\theta=(\alpha,\beta)\in\mathcal S$, Then $\bold E\,\theta:= \alpha^{-1}_*\beta^*\,\bold E$ defines a right action of $\mathcal S$ on $\EXT_\Phi(Q,H)$.\end{prop}
\begin{proof} One only needs to note that for any $\alpha\in\op{Aut}H$ and $\beta\in\op{End}Q$ that $\alpha_*\beta^*\,\bold E=\beta^*\alpha_*\,\bold E$ and this is trivial when one expresses $\bold E=\bold E_{f,\varphi}$.\end{proof}

$\Phi$ gives a well defined $Q$-module structure to $zH$.  The automorphism $\alpha:H\to H$ maps $zH\to zH$ and $\alpha$ and $\beta$ define homomorphisms $\alpha_*,\beta^*:H^2(Q,{}_\beta\,zH)\to H^2(Q,\,{}_\beta\,zH)$ and hence a well defined homomorphism $\theta^*=\alpha^{-1}_*\beta^*:H^2(Q,\,zH)\to H^2(Q,\,zH)$. It is easy to see that $(\theta'\theta)^*=\theta^*\theta'^*$ and hence a right action of $\mathcal S$ on $H^2(Q,zH)$, $\zeta\,\theta=\theta^*(\zeta)$.  $H^2(Q,zH)$ acts simply transitively on (the right) $\op{\mathcal{X}}_\Phi(Q,H)$ as follows.  If $\bold E=\bold E_{\varphi,f}\in\EXT_\Phi(Q,H)$ and $\zeta=[c]\in H^2(Q,zH)$ then $\be\cdot\zeta=\bold E_{\varphi,f+c}$. An easy calculation shows

\begin{prop} i) $\alpha_*(\bold E\,\zeta)=(\alpha_*\bold E)\,\alpha_*\zeta$, ii) $\beta^*(\bold E\,\zeta)=(\beta^*\bold E)\beta^*\zeta$ and hence if $\theta\in\mathcal S$, $\zeta\in H^2(Q,zN)$, $(\bold E\cdot\zeta)\theta=\bold E\theta\cdot\zeta\theta$.\end{prop} 



Recall that an extension $1\to H\to G\rar{\pi} Q\to 1$ is split if there exists a homomorphism $s:Q\to G$ with $\pi\,s=\text{id}$.

\begin{prop}\label{split} If $\bold E$ is a split extension then both $\alpha_*\bold E$ and $\beta^*\bold E$ are split extensions.\end{prop}
\begin{proof} That the pullback $\beta^*\bold E$ is split follows from the definition of the pullback.  If $\bold E=\bold E_{\varphi,f}$ and $s(q)=(\tau q,q)$ is a splitting for $\bold E$ the fact that $s$ is a homomorphism translates to the fact that
$$ \tau(qq')=\tau q+\varphi(q)\tau q'+f(q,q').   \qquad(*)$$
Defining $u(q)=(\alpha\tau q,q)$, $u$ is a homomorphism if and only if
$$ \alpha\tau qq'=\alpha\tau q+\alpha\varphi\alpha^{-1}(\alpha\tau q') +\alpha f(q,q') $$ which follows from $(*)$ by applying $\alpha$.\end{proof}

\begin{remark} Unfortunately this does not show that split extensions are always fixed points, but only that under the action of $\mathcal S$ on $\mathcal X_\Phi(Q,H)$ split extensions are always taken to split extensions.  If there exists a unique split extension in $\mathcal X_\Phi(Q,H)$, e.g. if $H$ is abelian, then we have a fixed point.  But for most non-abelian $H$ with non-trivial center there exist many split extensions in in a given outermorphism class $\Phi$.  For example if $H$ is the dihedral or quaternion group of order $8$ or $16$, $Q$ is $\Z/2$ or $\Z/4$ and $Q\to\op{Out}H$ the trivial homomorphism, then $\mathcal X_\Phi(Q,H)$ consists of two elements and both are split extensions.  Whether this is always true is yet to be decided.
\end{remark}

Now subtraction defines a map
$$s:\mathcal X_\Phi(Q,H)\times \mathcal X_\Phi(Q,H) \to H^2(Q,zH)$$ as follows.  Let $\be,\be'\in\EXT_\Phi(Q,H)$. Choose any basepoint $\bold E_0$.  Then $\bold E\equiv\bold E_0\zeta$, $\bold E'\equiv\bold E_0\zeta'$ with $\zeta,\zeta'\in H^2(Q,zH)$, and define $s(\be,\be')=\zeta-\zeta'\in H^2(Q,zH)$.  This is independent of $\bold E_0$.  Denote $s(\bold E,\bold E')=(\bold E-\bold E')$.  $\mathcal S$ acts (on the right) on both $\mathcal X_\Phi(Q,H)\times \mathcal X_\Phi(Q,H)$ and $H^2(Q,zH)$.

\begin{prop} \label{diff} i) $s$ is equivariant, i.e., if $\theta\in\mathcal S$ then $(\bold E-\bold E')\theta=(\bold E\theta-\bold E'\theta)$.
\par ii) $(\bold E-\bold E')=-(\bold E'-\bold E)$.
\par iii) $(\bold E-\bold E')+(\bold E'-\bold E'')=(\bold E-\bold E'')$.
\end{prop}
\begin{proof} i) If $\bold E_0$ is a base point and $\bold E=\bold E_0a$ and $\bold E'=\bold E_0b$ then $(\bold E-\bold E')=a-b$, $\bold E\theta=(\bold E_0a)\theta=(\bold E_0\theta)a\theta$ and similarly for $\bold E'\theta$.  Since $(\bold E\theta-\bold E'\theta)$ is independent of the base points $(\bold E\theta-\bold E'\theta)=a\theta-b\theta$. ii) and iii) are clear from the definition.\end{proof}

Define $\lambda_\be:\mathcal S\to H^2(Q,zH)$ by $\lambda_\be(\theta)=(\be-\be\theta)$.

\begin{prop} $\lambda_{\bold E}:\mathcal S\to H^2(Q,zH)$ is a derivation, that is $$\lambda_\bold E(\theta\theta')=\lambda_\bold E(\theta)\theta'+\lambda_\bold E(\theta').$$ \end{prop}
\begin{proof} Let $\bold E_0$ be a base point and $\bold E=\bold E_0a$, $\bold E\theta'=\bold E_0b$, $\bold E\theta\theta'=\bold E_0c$ with $a,b,c\in H^2(Q,zH)$.  Then

$$\begin{aligned}\lambda_\bold E(\theta')+\lambda_\bold E(\theta)\theta'-\lambda_\bold E(\theta\theta')
&=(\bold E-\bold E\theta')+(\bold E-\bold E \theta)\theta'-(\bold E-\bold E\theta\theta')\\
&=(\bold E-\bold E\theta')+(\bold E\theta'-\bold E \theta\theta')+(\bold E\theta\theta'-\bold E) \\
&= (a-b)+(b-c)+(c-a)=0.
\end{aligned}$$\end{proof}






Let $\op{Aut}(G,H)=\{\gamma\in\op{Aut}G\vv\gamma\vert H:H\to H\}$.  There exists an obvious homomorphism
$\text{res}:\op{Aut}(G,H)\to \mathcal S$ given by $\text{res}(\gamma)=(\gamma\vert H,\gamma')$ where $\gamma'\pi=\pi\gamma$.  Choosing a basepoint $\be_0$ one sees $\lambda_E\theta=0$ if and only if $\be\equiv\be\,\theta$ for $\theta\in\mathcal S$.  Therefore in view of  \ref{MT} the following are equivalent.
\begin{enumerate}\item $\theta$ extends to a mapping $\bold E\to\bold E$.  That is $\theta\in\text{image}\{\text{res}\}$.
\item $\theta\in\op{Iso}_\mathcal S(\bold E)$, the isotopy group of $\bold E$.
\item $\theta\in\ker \lambda_\be$.\end{enumerate}

\begin{thm}\label{cycle} Let $\bold E\in\EXT_\Phi(Q,H)$. There exists an exact sequence of groups 
$$   0\to \mathcal Z^1(Q,zH)\rar{\mu} \op{Aut}(G,H)\rar{res}\mathcal S\rar{\lambda_\be} H^2(Q,zH)  $$
where $Q$ acts on $zH$ via $\Phi$ and $\mathcal Z^1(Q,zH)$ are the cocycles with respect to this action.  The map $\mu$ is given by $\sigma\mapsto\varphi_\sigma$ where $\varphi_\sigma(g)=\sigma(\pi g)g$ and $\lambda_\be$ is a derivation.
\end{thm}

\begin{proof} All is clear except for the identification of the kernel of the map $\op{Aut}G\rar{res}\op{Iso}_\mathcal S(\bold E)\sub\mathcal S$ given by $\gamma\mapsto(\gamma\vert_H,\gamma')$ where $\pi\gamma=\gamma'\pi$. Assume $\bold E=\bold E_{f,\varphi}$.  By \ref{map} 
$\gamma(h,q)=(h+\sigma q,q)$ for some $\sigma:Q\to H$ satisfying
$$ \sigma q+\varphi(q)[h+\sigma q']+f(q,q')=[\varphi(q)h]+f(q,q')+\sigma(qq')\qquad (*). $$
But $(0,q)+(h,1)=(\varphi(q)h,q)$ and so $\gamma(0,q)+\gamma(h,1)=\gamma(\varphi(q)h,q)$ or
$$(\varphi(q)h+\sigma q,q)=(\sigma q,q)+(h,1)=(\sigma q+\varphi(q)h,q).$$
Therefore $\sigma:Q\to zH$ and $(*)$ says $\sigma$ is a cocycle.  This shows the map $\mathcal Z^1(Q,zH)\to\ker\text{res}$ given by $\sigma\mapsto\gamma_\sigma$ where $\gamma_\sigma(h,q)=(h+\sigma(q),q)=(\sigma(q)+h,q)=(\sigma(q),1)+(h,q)$ is onto.  (The second equality follows since $\sigma(q)$ is central.)  It is clearly a monomomorphism.
\par By choosing an appropriate (normalized) section $s:Q\to G$ we obtain an isomorphism $\rho:G\to E_{f,\varphi}$ given by $g=hs(\pi g)\mapsto (h,\pi g)$.  Then
$$ \varphi_\sigma(g)=\rho^{-1}\gamma\,\rho(g)=\rho^{-1}\gamma(h,\pi g)=\rho^{-1}\left(\sigma(\pi g),1)+(h,q)\right)=\sigma(\pi g)g.$$\end{proof}

\begin{remark} All the terms in the above exact sequence except for $\op{Aut}(G,H)$ depend only on $H$, $Q$ and the action of $Q$ on the center of $H$ and not on the extension $\be$.  The image of $\text{res}=\op{Iso}_{\mathcal S}\be$ clearly depends on $\be$. This is well illustrated by the groups $D_4$ and $Q_8$ which have center $H=\Z/2$ (and hence $\op{Aut}(G,H)=\op{Aut}G$ in both cases) and quotient $Q=\Z/2\times\Z/2$. Since $zH$ is a trivial $Q$-module $\mathcal Z^1(Q,zH)=H^1(Q,zH)=(\Z/2)^2$, we have an exact sequence
$$ 0\to \Z/2\times\Z/2\to  \op{Aut}G\to \op{Iso}_{\hat{\mathcal S}}\bold E \to 1.  $$
Now $H^2(Q,zH)\simeq(\Z/2)^3$.  Since $\Phi=1$, $$\mathcal S=\op{Aut}\Z/2\times\op{Aut}(\Z/2\times\Z/2)=GL_2(F_2)=S_3$$ acts on a set of order $8$ representing the extensions of $\Z/2$ by the Klein $4$-group. It is not difficult to see that there are $4$ orbits, two orbits with one element representing $(\Z/2)^3$ and $Q_8$ and two orbits with $3$ elements representing $D_4$ and $\Z/4\times \Z/2$.  Hence we have exact sequences
$$\begin{aligned} 0\to V\to &\op{Aut}D_4\to \Z/2\to 1  \\
0\to V\to&\op{Aut}Q_8\to S_3\to 1\end{aligned} $$
corresponding to the facts that $\op{Aut}D_4\simeq D_4$ and $\op{Aut}Q_8\simeq S_4$.
 \end{remark}

\section{$\op{Out}(G,H)$ and the basic exact sequence}

$\op{Inn}G$ is a normal subgroup of $\op{Aut}(G,H)$. Let $\op{Out}(G,H)=\op{Aut}(G,H)/\op{Inn}G$. 

Consider the map $u=\text{res}\vert_{\op{Inn}G}:\op{Inn}G\to\op{Aut}(G,H)\to\mathcal S$.  The image of this map is $\mathcal B=\{u(c_g)=(c_g\vert H,c_{\pi g})\vv g\in G\}$.

\begin{prop} i) $\mathcal B$ is normal in $\mathcal S$. \par ii) $\mathcal B\sub\op{Iso}_\mathcal S\bold E$. \par iii) $\ker u\simeq (C_GH\cap\pi^{-1}zQ)/zG$.\end{prop} 
\begin{proof} i) If $(\alpha,\beta)\in\mathcal S$ then $(\alpha,\beta)(c_g\vert H,c_{\pi g})(\alpha,\beta)^{-1}=(\alpha\,c_g\vert H\alpha^{-1},c_{\beta\pi\,g}).$  Now $(\alpha,\beta)\in\mathcal S$ means $\Phi(\beta\,\pi g)=\alpha\,\Phi(\pi g)\alpha^{-1}$.  But $\Phi(\pi g)=[c_g\vert H]$ and so if $g'\in G$ is a preimage of $\beta\pi\,g$ we have $[c_{g'}\vert H]=[\alpha\,c_g\vert H\alpha^{-1}]$.  Therefore $\alpha\,c_g\vert H\alpha^{-1}=c_{g'}\vert H\,c_h$ for some $h\in H$ and hence
$$ (\alpha\,c_g\vert H\alpha^{-1},c_{\beta\pi g})=(c_{g'}\vert H,c_{\beta\pi,g})(c_h\vert H,\text{id}).$$  
Clearly $(c_h\vert H,\text{id})=u(c_h)\in\mathcal B$ and since $\pi g'=\beta\pi g$, $(c_{g'}\vert H,c_{\beta\pi,g})=u(c_{g'})\in\mathcal B$.  
\par ii) This is just \ref{MT} since $(c_g\vert H,c_g,c_{\pi g}):\bold E\to \bold E$.
\par iii) $c_g\vert H=\text{id}$ if and only if $g\in C_GH$ and $c_{\pi g}=\text{id}$ if and only if $g\in \pi^{-1}zQ$.
\end{proof}

Let $\bar g\in C_GH\cap\pi^{-1}zQ$ and let $\sigma:G\to G$ given by $\sigma_{\bar g}(g)=[\bar g,g]=\bar gg\bar g^{-1}g^{-1}$.

\begin{prop} i) $\sigma_{\bar g}(gh)=\sigma_{\bar g}(g)$ and hence defines a map $\sigma:Q\to G$.
\par ii) $\sigma_{\bar g}(g)\in zH$.
\par iii) $\sigma_{\bar g}:Q\to zH$ is a derivation with respect to the conjugation action of $Q$ on $zH$.\end{prop}
\begin{proof} i) $\sigma_{\bar g}(gh)=\bar g gh\bar g^{-1}(gh)^{-1}=\bar g g\bar g^{-1}hh^{-1}g^{-1}=\sigma_{\bar g} g$.
\par ii) $\bar g\in \pi^{-1}zQ$ implies $\sigma_{\bar g} g\in H$.  Since $\bar g\in C_GH$ and $h,\,g^{-1}hg\in H$
$$ h\sigma_{\bar g}(g)=\bar ghg\bar g^{-1}g^{-1}=\bar gg(g^{-1}hg)\bar g^{-1}g^{-1}=(\bar gg\bar g^{-1}g^{-1})h=\sigma_{\bar g}(g)h.$$ 
\par iii) If $\pi g=q$, $\pi g'=q'$ then
$$ \sigma_{\bar g}(qq')=\bar g(gg')\bar g^{-1}(gg')^{-1}=\bar gg\bar g^{-1}g^{-1}g\bar gg'\bar g^{-1}g'^{-1}g^{-1}=\sigma_{\bar g}(q){}^q\sigma_{\bar g}(q').$$
\end{proof}
If we define $\sigma:C_G\cap\pi^{-1}zQ/zG\to Z^1(Q,zH)$ by $\sigma(g\,zG)=\sigma_g$ we see easily that we have a commutative diagram

$$\xymatrix{ 1\ar[r]&C_GH\cap zQ/zG\ar[r]^c\ar@{^{(}->}[d]_\sigma& \op{Inn}G\ar[r]^{res}\ar@{^{(}->}[d]&\mathcal B\ar[r]\ar@{^{(}->}[d]&1\\0\ar[r]&Z^1(Q,zH)\ar[r]^\lambda&\op{Aut}(G,H)\ar[r]^{res}&\op{Iso}_{\mathcal S}\bold E\ar[r]&1}.$$

If $\bar B=\text{Image}\,\sigma$, $Z^1(Q,zH)\supseteq\bar B\supseteq B^1(Q,zH)=\sigma(zH\,zG/zG)$ where $B^1(Q,zH)$ are the boundaries and we have an exact sequences
$$\begin{aligned} 0\to \bar H^1(Q,zH)&\rar{\lambda} \op{Out}(G,H)\rar{res}\op{Iso}_{\mathcal S}\bold E\to 1,\\
0\to (C_GH\cap\pi^{-1}zQ)&/zHzG\to H^1(Q,zH)\to \bar H^1(Q,zH)\to 0.\end{aligned}  $$

\begin{remark} If $H$ is centric in $G$ then $C_GH\cap\pi^{-1}zQ=zH$, $zG\sub zH$ and so $\bar H^1(Q,zH)=H^1(Q,zH)$.\end{remark}

$\op{Iso}_{\mathcal S}\bold E/\mathcal B\sub \mathcal S/\mathcal B\equiv \bar{\mathcal S}$.  $\bar{\mathcal S}$ does not necessarily act on $\mathcal X_{\Phi}(Q,H)$ since $\mathcal B$ does not necessarily act trivially.  Let $\mathcal O_{\bold E}$ denote the orbit of $\bold E\in\mathcal X_{\Phi}(Q,H)$ under the right action of $\mathcal S$, and for $g\in G$ let $\theta_g=(c_g\vert H,c_{\pi g})\in\mathcal S$.

\begin{prop} i) $\lambda_{\bold E}(\theta_g\theta)=\lambda_{\bold E}\theta$ for $\theta\in\mathcal S$ and hence defines a map $\lambda_{\bold E}:\bar{\mathcal S}\to H^2(Q,zH)$.
\par ii) $\bar{\mathcal S}$ acts (on the right) on $\mathcal O_{\bold E}$ and the map $\lambda_{\bold E}:\bar{\mathcal S}\to U\sub H^2(Q,zH)$.  $U$ is an $\bar{\mathcal S}$ invariant subgroup of $H^2(Q,zH)$ and $\lambda_\be$ is a derivation with respect to this action.
\par iii) $\op{Iso}_{\bar{\mathcal S}}\bold E=\op{Iso}_{\mathcal S}\bold E/\mathcal B.$
\end{prop}
\begin{proof} i) $\lambda_{\bold E}$ is a derivation and $\bold E\,\theta_g\equiv\bold E$.  Hence $\lambda_{\bold E}(\theta_g\theta)=\lambda_{\bold E}(\theta_g)\theta+\lambda_{\bold E}(\theta)=\lambda_{\bold E}\theta$.
\par ii) Since $\mathcal B$ is normal in $\mathcal S$ we have $\lambda_{\bold E}(\theta\,\theta_g)=\lambda_{\bold E}(\theta)$.  I.e, for all $\theta,\,\theta_g\in \mathcal S$
$$\bold E\theta-\bold E=\bold E(\theta\,\theta_g)-\bold E=(\bold E\theta)\theta_g-\bold E.$$.  But by choosing a base point in $\mathcal X_\Phi(Q,H)$ this easily implies $(\bold E\theta)\theta_g=\bold E\theta$, that is $\theta_g$ acts trivally on $\mathcal O_{\bold E}$.  $U=\{(\be'-\be'')\vv \be',\be''\in \mathcal O_{\be}\}$.  That the induced map $\lambda_E:\bar{\mathcal S}\to H^2(Q,zH)$ is a derivation is obvious.
\par iii) This is also obvious.\end{proof}

This gives the following.

\begin{thm}\label{basic} Let $\bar{\mathcal S}=\mathcal S/\mathcal B$ then $\bar{\mathcal S}$ acts on $\mathcal O_{\bold E}\sub \mathcal X_\Phi(Q,zH)$ and there exists an exact sequence
$$  0\to \bar H^1(Q,zH)\rar{\mu}\op{Out}(G,H)\rar{res}\bar{\mathcal S}\rar{\lambda_{\bold E}}H^2(Q,zH).  $$
$image\{res\}=\op{Iso}_{\bar{\mathcal S}}\bold E$ and $\bar H^1(Q,zH)\simeq H^1(Q,zH)/V$ with $V\simeq (C_GH\cap \pi^{-1}zQ)/zHzG$.

\end{thm} 

\section{Decomposition of $\bar{\mathcal S}$}

\begin{defn} $\op{Aut}_\Phi Q=\{\beta\in\op{Aut}Q\vv \Phi\beta=\Phi\}$.
Identify $\beta\in\op{Aut}_\Phi Q$ with $(1,\beta)\in\mathcal S$.\end{defn}

If $(\alpha,\beta)\in\mathcal S$ then $[\alpha\Phi(q)\alpha^{-1}]=\Phi(\beta q)$ and hence $[\alpha]\in N_{\op{Out}H}(\Phi Q)$ and it is clear we have a homomorphism $p:\mathcal S\to N_{\op{Out}H}(\Phi Q)$ given by $p(\alpha,\beta)=[\alpha]$.

\begin{lemma} i) $p(\mathcal B)=\Phi(Q)$ and hence we have a homomorphism 
$$p:\bar{\mathcal S}\to  N_{\op{Out}H}(\Phi Q)/\Phi Q.$$
\par ii) $\ker p=\mathcal B\,\op{Aut}_\Phi (Q)/\mathcal B\simeq \op{Aut}_\Phi Q/(\op{Aut}_\Phi Q\cap\mathcal B).$\end{lemma}
\begin{proof} i) $\Phi(Q)=\{[c_g\vert H]\vv g\in G\}=p(\mathcal B)$.
\par ii) If $[\alpha]=[c_g\vert H]$ then $\alpha=c_g\vert H\,c_h=c_{\bar g}\vert H$ for some $\bar g\in G$.  Then 
$$(\alpha,\beta)=(c_{\bar g}\vert H,c_{\pi\bar g})(\text{id},c_{\pi\bar g}^{-1}\beta)\in\mathcal B
\op{Aut}_\Phi Q.$$
The converse is obvious.\end{proof}

\begin{remark} i) If $H$ is centric in $G$ then $\op{Aut}_\Phi Q=\{\text{id}\}$ since $\Phi(\beta\,q)=\Phi(q)$ implies $\beta\,q=q$ since centric is equivalent to $\Phi$ being a monomorphism.  Hence in the centric case $p$ is a monomorphism.
\par ii) More generally, $p$ is a monomorphism if and only if $\op{Aut}_\Phi Q\sub\mathcal B$ if and only $\Phi\beta=\Phi$ implies $\beta=c_{\pi g}$ for some $g\in C_GH$.\end{remark}

In order to determine the image of $p$ let $Q'=Q/\ker\Phi$ with natural projection $\tau:Q\to Q'$.  $\Phi$ defines a monomorphism $\Phi':Q'\to\op{Out}H$ with $\Phi'\tau=\Phi$. Now $\Phi'(Q')=\Phi(Q)$ and if $[\alpha]\in N_{\op{Out}H}(\Phi Q)$ then there exist a (unique) homomorphism $\beta':Q'\to Q'$ with $[\alpha\Phi'(q')\alpha^{-1}]=\Phi'(\beta'q')$ for all $q'\in Q'$.  If $[\alpha]\in\text{image}\,p$ then there exists $\beta:Q\to Q$ inducing $\beta'$, i.e., with $\tau\beta=\beta'\tau$.  Conversely if $\beta$ exists then for all $q\in Q$ we have 
$$[\alpha\Phi q\alpha]=[\alpha\Phi'(\tau q)\alpha^{-1}]=\Phi'(\beta'\tau q)=\Phi'(\tau\beta q)=\Phi(\beta q)$$
and $(\alpha,\beta)\in\mathcal S$.  If $[\alpha]\in N_{\op{Out}H}(\Phi Q)$ determines the map $\beta':Q'\to Q'$ and $(c_g\vert H,c_{\pi g})\in \mathcal B$ then $[\alpha\,c_g\vert H]$ determines the map $\beta'c_{\pi g}:Q'\to Q'$.

\begin{lemma} $\beta'$ lifts to a map $Q\to Q$ if and only if $\beta'c_{\pi g}$ lifts to a map $Q\to Q$.\end{lemma}
\begin{proof} The map $\pi:G\to Q$ induces the following commutative diagram
$$\begin{CD}   1 \to HC_GH/H @>>> G/H @>>> G/HC_GH\to 1  \\
         @V\pi V\simeq V   @V\pi V\simeq V  @V\pi V\simeq V &\\
1\to \ker\Phi@>>> Q@>\tau>> Q'\to 1.\end{CD}$$
Therefore $c_{\pi g}:Q'\to Q'$ lifts to $c_{\pi g}:Q\to Q$ and the result follows.
\end{proof}

If $\ker\Phi$ is abelian there is by \cite{R} a well defined cohomolgy class $\Delta(\beta ')\in H^2(Q,\ker\Phi)$ such that $\beta'$ lifts to a map $\beta:Q\to Q$ if and only if $\Delta(\beta ')$ is the base point of $H^2(Q,\ker\Phi)$.  Hence

\begin{prop} If $\ker\Phi$ is abelian there exists an exact sequence
$$    \bar{\mathcal S}\rar{p}N_{\op{Out}H}(\Phi Q)/\Phi Q\rar{\Delta}H^2(Q,\ker \Phi).$$
\end{prop}

\begin{remark} In the centric case, $C_GH\sub H$, we have from remarks 3 and 4, i) $H^1(Q,zH)=\bar H^1(Q,zH)$ and  ii) $\ker p\simeq \op{Aut}_\Phi Q/(\op{Aut}_\Phi Q\cap\mathcal B)=\{\id\}$.  Clearly $H^2(Q,\ker\Phi)=\{*\}$ and so the exact sequence \ref{basic} becomes the exact sequence of \cite O
$$ 0\to H^1(Q,zH)\rar{\mu}\op{Aut}(G,H)\rar{res}N_{\op{Out}H}(\Phi Q)/\Phi Q\rar{\lambda_\be}H^2(Q,zH)$$
and with $\lambda_\be$ a derivation.\end{remark}

\section{Solvability of $\op{Aut}(G,H)$}

\begin{prop}\label{presolv} If $H$, $\op{Aut}_\Phi Q$ and $N_{\op{Out}H}(\Phi Q)$ are solvable so is $\mathcal S$.  Conversely if $\mathcal S$ is solvable, then $H$, $\op{Aut}_\Phi Q$ and $C_{\op{Out}H}(\Phi Q)$ are solvable.\end{prop}
\begin{proof} If $p: \mathcal S\to N_{\op{Out}H}(\Phi Q)$ is given by $p(\alpha,\beta)=[\alpha]$, then $\ker p=\{(c_h,\beta)\vv h\in H\}=\mathcal U\op{Aut}_\Phi Q$ where $\mathcal U=\{(c_h,\id)\vv h\in H\}$ is a normal subgroup of $\mathcal S$.  Since $\mathcal U$ is solvable if $H$ is, we see
from the hypothesis, $\mathcal S$ is solvable if and only if $\op{Aut}_\Phi Q$ is.

Conversely if $\mathcal S$ is solvable then so is $\mathcal I=\{(c_h,\id)\vv h\in H\}\sub\mathcal S$ and therfore also $H$ since $H/zH\simeq\mathcal I$.  Clearly $C_{\op{Out}H}(\Phi Q)\sub\mathcal S$ as the set $\{(\alpha,\id)\}$.
\end{proof}

\begin{thm}\label{solv} Let $1\to H\to G\to Q\to 1$ be an exact of groups with associated homomorphism $\Phi:Q\to\op{Out} H$.  Suppose 1) $H$ is solvable, 2) $H$ is characteristic in $G$. 3) $N_{\op{out}H}(\Phi Q)$, the normalizer of $\Phi(Q)\sub\op{Out} H$, is solvable, 4) $\op{Aut}(\ker\Phi)$ is solvable.  Then $\op{Aut}(G,H)$ is solvable. 
\par For each condition there exists an extension where the remaining three conditions are true and $\op{Aut}G$ is not solvable and for each condition except the first there exists an extension where exactly one of the remaining conditions is false and $\op{Aut}G$ is solvable. \end{thm}
\begin{proof} From \ref{cycle} we see that $\op{Aut}(G,H)$ is solvable if and only if $\op{Iso}_{\mathcal S}\be\sub\mathcal S$ is solvable. Since $H$ is characteristic in $G$, $\op{Aut}(G,H)=\op{Aut}G$.  To show $\mathcal S$ is solvable we need to show that $\op{Aut}(\ker\Phi)$ being solvable implies $\op{Aut}_\Phi Q$ is solvable.

If $\beta\in\op{Aut}_\Phi Q$ then $\Phi\,\beta=\Phi$ and $\beta$ restricts to an automorphism of $\ker\Phi$ and so we have a homomorphism $\op{Aut}_\Phi Q\to\op{Aut}(\ker\Phi)$ with kernel $K$.  Again by the hypothesis, $\op{Aut}_\Phi Q$ is solvable if and only if $K$ is.

$\beta\in K$ if and only if $\Phi\beta=\Phi$ and $\beta\vert_{\ker\Phi}=\text{Id}$.  If $\lambda(q)=\beta(q)q^{-1}$ then $\Phi(\lambda q)= 1$ and so $\lambda:Q\to\ker\Phi$.  $\beta(q)=q$ for $q\in\ker\Phi$ and so $\lambda\vert{\ker\Phi}=1$.  $\beta$ a homomorphism if and only if $\lambda$ is a $1$-cocyle, that is $\lambda(qq')=\lambda(q){}^q\lambda(q')$ for all $q,q'\in Q$ where $Q$ is acting by conjugation.  If $q\in Q$ and $q'\in\ker\Phi$ then $\lambda(qq')=\lambda(q)$ and so $\lambda$ is constant on left cosets and hence on right cosets since $\ker\Phi$ is normal in $Q$.  Therefore for $q\in Q$ and $q'\in\ker\Phi$
$$ \lambda(q)=\lambda(q'q)={}^{q'}\lambda(q)=q'\lambda(q)q'^{-1}. $$
Therefore $\lambda(q)\in z\ker\Phi$ and defines a cohomology class $[\lambda]\in H^1(Q,z\ker\Phi)$. Let $\sigma(\beta)=[\lambda]$.  It is easily seen $\sigma:K\to H^1(Q',z\ker\Phi)$ is a homomorphism. If $\sigma(\beta)=[0]$ then $\lambda(q)=w{}^qw^{-1}$ for some $w\in z\ker\Phi$. But then $\beta(q)=\lambda(q)q=wqw^{-1}q^{-1}q=c_w(q)$.  Hence we have an exact sequence $z\ker\Phi\to K\to H^1(Q',z\ker\Phi)$ and $K$ is solvable.




\par Let $p$ and $n!$ be relatively prime and $G=H\times Q=A_5\times\Z/p$. Then 1) is false, 2), 3) and 4) are true and $\op{Aut}G\simeq\op{Aut}A_5\times \Z/(p-1)$ is not solvable.
\par If $G=(\Z/2)^3=H\times Q=(\Z/2)^2\times\Z/2$ then 2) is false 1), 3), and 4) are true and $\op{Aut}(G)\simeq \op{GL}_3(\bold F_2)$ is not solvable.
\par For 3) let $H=(\Z/2)^3$, $Q=\Z/3$ and $G=H\times Q$. Then 3) is false, 1) 2) and 4) are true and $\op{Aut}G$ is non-solvable.
\par For 4) just interchange the roles of $H$ and $Q$ in example $3$. 

\par As for the necessity of the above conditons. If $\op{Aut}G$ is solvable then $G$ and therefore $H$ must be solvable. Certainly being characteristic is not a necessary condition as $\Z/2\times\Z/2$ has a solvable automorphism group but has no non-trivial characteristic subgroups.
\par Let $H$ be the small group $(25,2)$.  This is a solvable group whose automorphism group $A$ (small group $(480,218)$) is not solvable.  $A$ contains a normal subgroup $H$ of order $2$.  Let $G$ be the semi-direct product of $H$ and $Q=\Z/2$, where $\Z/2$ acts via $H$ (small group $(50,3)$).  Then $\op{Aut}G$ (small group $(80,30)$) is solvable and $\ker\Phi=(1)$.  Clearly $H$ is characteristic in $G$ since it is the Sylow $5$-subgroup.  Hence $H\to G\to\Q$ satisfies all the conditions of the theorem except for $3$ and has $\op{Aut}G$ solvable.  Hence condition $3$ is not necessary.    
\par Let $Q$ be the small group $(24,13)$.  $Q$ has Sylow $2$-subgroup $P=(\Z/2)^3$ with non-solvable automorphism group $GL_3(\F_2)$.  $Q$ is the split extension of $(\Z/2)^3$ by any element of order $3$ in $\op{Aut}P$. Let $H=\Z/7$ and $\Phi':G/P=\Z/3\to \op{Aut} H$ be any monomorphism.  Let $H$ be the metacylic group of order $21$ and $\Z/7\to Y\to Q$ the pullback of the extension $\Z/7\to H\to\Z/3$ by the natural map $\pi:Q\to Q/P$.  Then this extension has $\Phi=\Phi'\pi$ with kernel $\Phi=P$ with non-solvable automorphism group.  $\Z/7$ is solvable and since $\op{Aut}\Z/7$ is abelian, $\Phi(Q)=\Phi'(\Z/3)$ has solvable normalizer.  Also $\Z/7$ is characteristic in $Y$ since it is the Sylow $7$-subgroup.  Therefore this extension satisfies conditions $1$, $2$ and $3$ of the theorem but not $4$.  We may also consider $Y$ as an extension of $P$ with quotient $H$.  If $\Psi:H\to\op{Out}P=\op{Aut}P$ is the homomorphism associated to this extension, then $H_{\op{Out}P}(\Psi H)$ is of order $6$ and solvable.  $\op{Aut}H$ is a solvable group of order $42$.  Therefore the extension $P\to Y\to H$ satisfies all the conditions of the theorem and $\op{Aut}Y$ is solvable.  Therefore condition $4$ of the theorem is not necessary.\end{proof}

\begin{remark} (1) Since $\op{Aut}_\Phi Q\sub\op{Aut}Q$ we could replace conditon 4) of the above theorem by the condition $\op{Aut}Q$ be solvable.
It is not at all clear if either one of these conditions implies the other.
\par (2) With slightly more work one can show there is an exact sequence
$$ 0\to H^0(Q,A)\rar{res} H^0(\ker\Phi,A)\to K\rar{\sigma}H^1(Q,A)\rar{res} H^1(\ker\Phi,A).  $$ with $A=z\ker\Phi$.\end{remark}

\section{Herbrand quotient}

In order to obtain the exact sequence in \ref{basic} one starts with the exact sequence in \ref{cycle} and then divides by the sequence
$$  1\to \mathcal L\to \op{Inn}G\rar{res}\mathcal B\to 1  $$
where $\mathcal L\simeq (C_GH\cap\pi^{-1}zQ)/zG$.  Instead consider the exact sequence and map of exact sequences
\xymatrix{   1 \ar[r] & \mathcal K \ar[r]\ar[d]^a & \op{Aut}_HG\ar[r]\ar[d]^b &\op{Inn}H\ar[r]\ar[d]^c & 1 \\
1 \ar[r] &\mathcal L \ar[r] & \op{Inn}G\ar[r] &\mathcal B\ar[r]& 1}

\noindent where $\mathcal K\simeq zH/(H\cap zG)\simeq zHzG/zG$.  All groups are normal subgroups and the vertical maps are monomorphisms.
Dividing the exact sequence in \ref{cycle} by these two sequences gives the following.

\xymatrix{  0\ar[r]&\text{coker}\, a\ar[r]\ar[d]&\text{coker}\,b\ar[r]\ar[d]&\text{coker}\,c\ar[r]\ar[d] &1 \\
 0\ar[r] &H^1(Q,zH) \ar[r]\ar[d]^a & \op{Out}'(G,H)\ar[r]\ar[d]^b &\op{Iso}_{\hat{\mathcal S}}\be\ar[r]\ar[d]^c & 1 \\
0\ar[r] &\bar H^1(Q,zH) \ar[r] & \op{Out}(G,H)\ar[r] &\op{Iso}_{\bar{\mathcal S}}\be\ar[r]& 1 }
\noindent where $\hat{\mathcal S}=\mathcal S/\op{Inn}H\sub\op{Out}H\times\op{Aut}Q$ and $\op{Out}'(G,H)=\op{Aut}(G,H)/\op{Aut}_HG$.

\begin{thm} Suppose $\be:1\to H\to G\to Q\to 1$ is an extension with $G$ is a finite group.  Then
$$  \vert\op{Aut}(G,H)\vert=\frac{\vert H^1(Q,zH)\vert}{\vert H^0(Q,zH)\vert\,\vert \mathcal O_{\hat{\mathcal S}}\bold E\vert}\,\vert H\vert\,\vert\hat{\mathcal S}\vert$$
where $\mathcal O_{\hat{\mathcal S}}\bold E$ is the orbit of $\bold E$ in $\mathcal{X}_\Phi(Q,H)$.  Moreover $\vert\mathcal O_{\hat{\mathcal S}}\bold E\vert\leq \vert H^2(Q,zH)\vert$.
\end{thm} 
\begin{proof} The above diagram shows $\vert\op{Out}(G,H)\vert\vert\text{coker}\,b\vert=\vert H^1(Q,zH)\vert\vert\op{Iso}_{\hat{\mathcal S}}\be\vert$.  Since $\text{coker}\,b\simeq \op{Inn}G/\op[Aut_HG$ and $\op{Out}(G,H)\simeq\op{Aut}(G,H)/\op{Inn}G$, we have
$$ \vert\op{Aut}(G,H)\vert=\vert H^1(Q,zH)\vert\vert\op{Iso}_{\hat{\mathcal S}}\be\vert\vert\op{Aut}_HG\vert.$$
Since $\op{Aut}_HG\simeq H/zG\cap H$, $H^0(Q,zH)=zG\cap H$ and $\vert\mathcal O_{\hat{\mathcal S}}\bold E\vert\vert\op{Iso}_{\hat{\mathcal S}}\be\vert=\vert\hat{\mathcal S}\vert$ the result follows.  Since $\mathcal X_\Phi(Q,H)$ and $H^2(Q,zH)$ are bijectively equivalent it follows $\vert\mathcal O_{\hat{\mathcal S}}\bold E\vert\leq \vert H^2(Q,zH)\vert$.
\end{proof}

The exact sequence $0\to H^1(Q,zH)\to \op{Out}'(G,H)\to \op{Iso}_{\hat{\mathcal S}}\be\to 1 $ and the fact that $\op{Out}'(G,H)=\op{Aut}(G,H)/\op{Aut}_HG$ gives the following.

\begin{thm} Suppose $\be:1\to H\to G\to Q\to 1$ is an extension.  Then there exists a normal series $\op{Aut}G=A_0\triangleright A_1\triangleright A_2\triangleright A_3=(1)$ with 
$$A_2/A_3\simeq H/H^0(Q,zH)\qquad A_1/A_2\simeq H^1(Q,zH)\qquad A_0/A_1\simeq\op{Iso}_{\hat{\mathcal S}}\bold E. $$\end{thm}
\begin{proof} $A_1=\ker\{\op{Aut}(G,H)\to\op{Out}'(G,H)\to\op{Iso}_{\hat{\mathcal S}}\be\}$ and $A_2=\ker\{\op{Aut}(G,H)\to\op{Out}'(G,H)\}=\op{Aut}_HG$.\end{proof}

\end{document}